\documentclass{article}
\usepackage[centering,total={410pt,610pt}]{geometry}
\usepackage{graphicx}
\usepackage[titletoc]{appendix}
\usepackage{tikz}
\usetikzlibrary{calc}
\usepackage{pgfplots}
\pgfplotsset{compat=newest}
\usepackage{appendix}
\usepackage{tikz}
\usepackage{enumerate}

\usepackage{amsmath}
\usepackage{amssymb}
\usepackage{amsthm}
\newtheorem{theorem}{Theorem}[section]
\newtheorem{lemma}[theorem]{Lemma}

\theoremstyle{definition}

\theoremstyle{remark}

\newtheorem{proposition}[theorem]{Proposition}

\usepackage{float}
\usepackage{hyperref}
\hypersetup{colorlinks,citecolor=black,filecolor=black,linkcolor=black,urlcolor=black}
\floatstyle{ruled}
\newfloat{algorithm}{htbp}{loa}
\floatname{algorithm}{\smallskip Algorithm\smallskip}

\numberwithin{equation}{section}
\allowdisplaybreaks

\usepackage{color}

\newcommand{\R}{\mathbb{R}}

\title{Optimal Nonergodic Sublinear Convergence Rate of Proximal Point Algorithm for Maximal Monotone Inclusion Problems}
\author{%
Guoyong Gu%
\thanks{Department of Mathematics, Nanjing University, Nanjing, 210093, China.}
\thanks{Email: ggu@nju.edu.cn. This author was supported by the NSFC grant 11671195.}
\and
Junfeng Yang\footnotemark[1]
\thanks{Email: jfyang@nju.edu.cn. This author was supported by the NSFC grant  11771208.}
}
\date{}

\begin{document}
\maketitle

\begin{abstract}
  We establish the optimal nonergodic sublinear convergence rate of the proximal point algorithm for maximal monotone inclusion problems.
  First, the optimal bound is formulated by the performance estimation framework, resulting in an infinite dimensional nonconvex optimization problem, which is then equivalently reformulated as a finite dimensional semidefinite programming (SDP) problem. By constructing a feasible solution to the dual SDP, we obtain an upper bound on the optimal nonergodic sublinear rate. Finally, an example in two dimensional space is constructed to provide a lower bound on the optimal nonergodic sublinear rate. Since the lower bound provided by the example matches exactly the upper bound obtained by the dual SDP, we have thus established the worst case nonergodic sublinear convergence rate which is optimal in terms of both the order as well as the constants involved. Our result sharpens the understanding of the fundamental proximal point algorithm.

\bigskip

\end{abstract}

\section{Introduction}
\label{Sec:Introduction}

The term ``proximal point'' was originally coined by Moreau \cite{Mor65bsmf,Mor62} and was firstly introduced to solve optimization problems by Martinet \cite{Mar70,Mart72}. Later, the Proximal Point Algorithm (PPA) was refined by Rockafellar in \cite{Roc76a} for solving maximal monotone operator inclusion problems. Since then, PPA has been
playing fundamental roles in the understanding, design and analysis of optimization algorithms. It was shown in \cite{Roc76b} that the classical method of multipliers of Hestenes \cite{Hes69} and Powell \cite{Pow69} is a dual application of the PPA. Similarly, it was shown in \cite{EB92} that
the Douglas-Rachford operator splitting method \cite{DR56,LM79} is also an application of the PPA to a special splitting operator.
As for the rate of convergence, it was shown in \cite{Roc76a} that PPA converges at least linearly under some regularity conditions, e.g., Lipschitz continuity of the inverse operator at the origin and monotonically nondecreasingness of the proximal parameters.
For minimizing a proper lower semicontinuous convex function, nonasymptotic $O(1/N)$ sublinear convergence rate measured by function values has been established in \cite{Gul91sicon}, where $N$ denotes the iteration counter. See also \cite{Gul92siopt,BT09} for some accelerated proximal-point-like methods designed for solving convex optimization problems by using Nesterov type acceleration technique \cite{Nest83}.
%
For maximal monotone operator inclusion problems, an $O(1/N)$ sublinear convergence rate result measured by fixed point residual was derived in \cite[Proposition 8]{BL78} for the original PPA without regularity assumptions. The same nonasymptotic convergence rate was recently established in  \cite[Theorem 3.1]{HY12c} for the Douglas-Rachford operator splitting method, a generalization of PPA to treating the sum of two maximal monotone operators.


In this paper, we investigate the optimal nonergodic sublinear convergence rate of PPA for maximal monotone inclusion problems.
Our idea is to formulate the optimal rate by using the performance estimation framework originally proposed by Drori and Teboulle \cite{DT14mp,DT16mp}
and recently refined by Kim and Fessler ~\cite{KF16MPA,KF17JOTA,KF18SIOPTa,KF18SIOPTb,KF18arxiv} and Taylor et al.~\cite{THG17mp,THG17siopt,dKGT17ol,THG18jota}.
%
%
The performance estimation problem is expressed as an infinite dimensional nonconvex optimization problem, which is then equivalently
reformulated as a finite dimensional convex SDP. By constructing a dual feasible SDP solution, we are able to obtain an upper bound on the optimal rate. We then construct a two dimensional example to show that the upper bound obtained from the dual SDP feasible solution is nonimprovable, neither in the order nor in any of the constants involved.
Specifically, we show that the optimal rate is ${1 \over (1 + {1\over N})^N (N+1)}$ (when the underlying Euclidean space has dimension greater or equal to $2$), where $N$ denotes the iteration counter.
Compared to \cite[Proposition 8]{BL78} and \cite[Theorem 3.1]{HY12c}, the improvement here is approximately a constant factor of $\exp(-1)$. This improvement, though might be insignificant in practice,   eliminates the gap between the known bound and the exact worst case bound. Recently, an alike analysis for convex composite optimization problems has been given in \cite{THG17siopt}, where the quality of an approximate solution is measured by function value residual. In this paper, our focus is the more general maximal monotone inclusion problems, and the quality measure is the fixed point residual, see, e.g., \cite{BL78,HY12c,DavY16sc,DavY17mor}. After a short release of our work on arxiv, an accelerated PPA, which achieves $O(1/N^2)$ rate of convergence measured by fixed point residual also, is constructed in \cite{Kim19} by using the same performance estimation idea \cite{DT14mp,DT16mp}.

\subsection{Notation}
In this paper, we reside ourselves in the $n$-dimensional Euclidean space $\R^n$, with inner product denoted by $\langle\cdot,\cdot\rangle$ and the induced norm $\|\cdot\| = \sqrt{\langle\cdot,\cdot\rangle}$, though all the analysis can be easily extended to any finite dimensional real Euclidean spaces.  The superscript ``$T$'' denotes the matrix or vector transpose operation. For a matrix $M$, integers $p$, $q$, $s$ and $t$ such that $p<q$ and $s<t$, we let $M(p:q,s:t)$ be the submatrix of $M$ located in between the $p$-th and the $q$-th rows and the $s$-th and the $t$-th columns.
Similarly, $M(:,s:t)$ (resp., $M(p:q,:)$) represents submatrix of $M$ containing the $s$-th to the $t$-th columns (resp., the $p$-th to the $q$-th rows). The $(i,j)$-th element of $M$ is denoted by $M_{i,j}$. For two matrices $A$ and $B$ of the same size, we let $\langle A, B\rangle = \sum_{i,j} A_{i,j}B_{i,j}$ be the trace inner product. The set of real symmetric matrices of order $n$ is denoted by $\mathbb{S}^n$. The fact that a matrix $A\in\mathbb{S}^n$ is positive semidefinite is denoted by $A\succeq 0$.
The set of maximal monotone operators on $\mathbb{R}^n$ is denoted by ${\cal M}$.
The graph of a set-valued operator $T$ on $\R^n$, i.e., the set $\{(x,y)\in\R^n\times\R^n \mid y\in T(x)\}$, is denoted by $\text{graph}(T)$. Both the identity matrix of appropriate order and the identity operator will be denoted by $I$.
Other notation will be specified later.

\subsection{Organization}
The rest of this paper is organized as follows.
In section \ref{sc:PPA_pep}, we present PPA and the Performance Estimation Problem (PEP).
Then, the PEP is equivalently reformulated as an SDP in section \ref{sc:reform_pep}.
An upper bound on the worst case sublinear rate is given in section \ref{sc:upper_bound} via constructing a dual feasible SDP solution,
followed by a lower bound given in section \ref{sc:example} via constructing an example.
Finally, some concluding remarks are given in section \ref{sc:conclusions}.

\section{PPA and PEP}\label{sc:PPA_pep}
Let $A\in {\cal M}$ be a maximal monotone operator.
A fundamental problem in convex analysis and related fields is to find a zero of $A$, i.e., find $w^*\in\mathbb{R}^n$ such that
$0\in A(w^*)$, and PPA is a primary method for solving this problem \cite{Roc76a}. Let $\lambda > 0$ be a constant. The resolvent operator of $A$ is defined by $J_{\lambda A} = (I + \lambda A)^{-1}$. Minty \cite{Minty62} first proved that the resolvent operator of any maximal monotone operator is single-valued and everywhere defined. Denote the set of zeros of $A$ by $A^{-1}(0)$. It is elementary to show that $w^*\in A^{-1}(0)$ if and only if $w^* = J_{\lambda A}(w^*)$. The PPA is in fact an iteration scheme constructed based on this fixed point equation. Given $w^0\in\mathbb{R}^n$, PPA generates a sequence of points $\{w^k\}_{k=1}^\infty$ via
\begin{equation}\label{PPA}
  w^{k+1} = J_{\lambda A}(w^k), \quad k = 0,1,2,\ldots
\end{equation}
Let $N > 0$ be an integer. After $N$ iterations, $w^N$ is generated and its quality is naturally measured by $\|e(w^N,\lambda)\|$, where
\begin{equation}\label{def:e}
e(w,\lambda) := w - J_{\lambda A}(w).
\end{equation}
The quantity  $\|e(w^N,\lambda)\|$ is referred to as fixed point residual and has been frequently used in the literature, see, e.g., \cite{DavY16sc,DavY17mor}. It was shown in \cite[Eq. (3.1)]{HY12c} that, for fixed $\lambda >0$, $\|e(w^N,\lambda)\|$ is monotonically nonincreasing, based upon which a worst-case bound $1/(N+1)$ was established. Specifically, for all $N>0$, it holds that
\begin{eqnarray}\label{bnd-HY12c}
{\|e(w^N,\lambda)\|^2 \over \|w^0-w^*\|^2} \leq \frac{1}{N+1}.
\end{eqnarray}
In fact, the focus of \cite{HY12c} is the Douglas-Rachford operator splitting method \cite{LM79}. For PPA, the bound \eqref{bnd-HY12c} was first established in \cite[Proposition 8]{BL78}.
The result \eqref{bnd-HY12c} is nonasymptotic because it holds for all $N>0$. It is also referred to as a nonergodic result since the latest point $w^N$ is used on the left hand side of \eqref{bnd-HY12c}, instead of the average of all history points.
The main contribution of this work is to improve the bound ${1 \over N+1}$ on the right hand side of \eqref{bnd-HY12c} to its optimal value ${1 \over (1 + {1\over N})^N(N+1)}$. For this purpose, we introduce the performance estimation framework recently developed in \cite{DT14mp,THG17mp,THG17siopt,dKGT17ol,THG18jota}.

Consider the following class of operators
\[
{\cal P} = \{ A \in {\cal M} \mid A^{-1}(0) \neq \emptyset\}.
\]
Let $N>0$ be an integer, $\lambda >0$ be a parameter, and $A\in {\cal P}$. Initialized at $w^0\in\R^n$, the first $N$ iterations of PPA generates a unique sequence of points $\{w^1,\ldots, w^N\}$. We denote this procedure by
\[
\{w^1,    \ldots, w^N\} = \text{PPA}(A,w^0,\lambda,N).
\]
The quantity on the left hand side of \eqref{bnd-HY12c} is dependent on $(A,w^0,w^*,\lambda,N)$. For convenience, we define
\begin{equation}\label{def:epsilon}
  \varepsilon(A,w^0,w^*,\lambda,N) := \frac{\|e(w^N,\lambda)\|^2}{\|w^0 - w^*\|^2} = \frac{\|w^N - w^{N+1}\|^2}{\|w^0 - w^*\|^2},
\end{equation}
where $w^{N+1} = J_{\lambda A}(w^{N})$.
In the performance estimation framework,
the exact worst case complexity bound or optimal convergence rate, denoted by $\zeta(N)$, is formulated as
the following PEP:
\begin{eqnarray}\label{PEP0}
\zeta(N) :=
  \sup_{A, w^0, w^*, \lambda}
  \left\{\varepsilon(A,w^0,w^*,\lambda,N)
  \left|
  \begin{array}{l}
    A \in {\cal P}, \;     w^0\in\mathbb{R}^n, \; 0 \in A(w^*), \; \lambda > 0,  \medskip \\
     \{w^1,    \ldots, w^N \} = \text{PPA}(A,w^0,\lambda,N)
  \end{array}
  \right.
  \right\}.
\end{eqnarray}
In the following, we argue that one can always set $\lambda=1$, $w^*=0$ and $w^0\in\R^n$ such that $\|w^0\|=1$ in \eqref{PEP0} without affecting the value of $\zeta(N)$.

{\bf Fact 1}: For any $\lambda > 0$ and operator $T\in {\cal M}$, define $T_\lambda = T\circ (\lambda I)$. Then, $T\in {\cal P}$ if and only if $T_{\lambda}\in {\cal P}$, and $u = J_{\lambda T} z$ if and only if $u/\lambda = J_{T_\lambda}(z/\lambda)$.

Based on Fact 1, it is easy to show that the sequence $\{w^1,\ldots,w^N,w^{N+1}\}$ generated by  $\text{PPA}(A,w^0,\lambda,N+1)$ also satisfies
\[
\{w^1/\lambda,\ldots,w^N/\lambda,w^{N+1}/\lambda\} = \text{PPA}(A_{\lambda},w^0/\lambda,1,N+1).
\]
Furthermore, it is apparent that $0 \in A(w^*)$ if and only if $0 \in A_\lambda(w^*/\lambda)$.
As a result, we have
\[\varepsilon(A,w^0,w^*,\lambda,N) = \varepsilon(A_\lambda,w^0/\lambda,w^*/\lambda,1,N).\]
Since $A\in {\cal P}$ if and only if $A_{\lambda}\in {\cal P}$ and the focus is the worst case bound $\zeta(N)$, it is thus without loss of generality to assume $\lambda = 1$.

{\bf Fact 2}: For any $\gamma > 0$ and operator $T\in {\cal M}$, define $T^\gamma = (\frac{1}{\gamma}I) \circ T \circ (\gamma I)$. Then, $T\in {\cal P}$ if and only if $ T^\gamma\in {\cal P}$, and $u = J_{T} z$ if and only if $u/\gamma = J_{ T^\gamma} (z/\gamma)$.

Based on Fact 2, it is easy to show that the sequence $\{w^1,\ldots,w^N,w^{N+1}\}$ generated by  $\text{PPA}(A,w^0,1,N+1)$ also satisfies
\[
\{w^1/\gamma,\ldots,w^N/\gamma,w^{N+1}/\gamma\} = \text{PPA}(A^\gamma,w^0/\gamma,1,N+1).
\]
Furthermore, $0\in A(w^*)$ if and only if $0\in A^\gamma(w^*/\gamma)$, and  $A\in {\cal P}$ if and only if $A^\gamma  \in {\cal P}$.
These together imply that, for any $\gamma > 0$, $(A, w^0, w^*)$ (together with $\lambda = 1$) is a solution of \eqref{PEP0} if and only if
$(A^\gamma, w^0/\gamma, w^*/\gamma)$ is a solution.
Therefore, it is also without loss of generality to assume $\|w^0 - w^*\| = 1$.

{\bf Fact 3}: For any operator $T\in {\cal M}$ and $u,v\in \R^n$. Define $T' = T(\cdot + u) + v$. Then, $T\in{\cal P}$ if and only if $T'\in{\cal P}$. Furthermore, $J_{T'} = J_T(\cdot + u - v) - u$.

Based on Fact 3, it is easy to show that $w^* = J_A(w^*)$ if and only if $0 = J_{A'} (0)$, where $A' = A(\cdot+w^*)$. Therefore, it is without loss of generality either to assume $w^* = 0$.

In the following, we always assume $\lambda = 1$, $w^*=0$ and $\|w^0-w^*\| = \|w^0\| = 1$.
As a result, the PEP given in \eqref{PEP0} can be simplified to
\begin{eqnarray}\label{PEP1}
\zeta(N) :=
  \sup_{A, w^0}
  \left\{ \|w^N -  w^{N+1}\|^2
  \left|
  \begin{array}{l}
    A \in {\cal P}, \;     \|w^0\| = 1, \; 0 \in A(0),   \medskip \\
     \{w^1,    \ldots, w^{N+1}\} = \text{PPA}(A,w^0,1,N+1)
  \end{array}
  \right.
  \right\}.
\end{eqnarray}
The PEP given in \eqref{PEP1} is an infinite dimensional nonconvex optimization problem due to the presence of the constraint $A\in {\cal P}$.
This seemingly makes it intractable. However, this is not the case. In fact, we can reformulate \eqref{PEP1} equivalently as a finite dimensional convex SDP by using the maximal monotone operator interpolation and extension theorem \cite[Fact 1]{RTBG18}\cite[Theorem 20.21]{BC17book}, which also serves as the basis of the list of recent works  \cite{THG17mp,THG17siopt,dKGT17ol,THG18jota}.

\section{SDP reformulation of PEP}\label{sc:reform_pep}
\subsection{Operator interpolation}\label{sc:OperInter}
Let $K$ be an index set and ${\cal Q}$ be a set of operators on $\R^n$. A set $\{(x_j,y_j)\}_{j\in K} \subseteq \R^n\times \R^n$ is said to be ${\cal Q}$-interpolable if there exists an operator $T\in{\cal Q}$  such that $\{(x_j,y_j)\}_{j\in K} \subseteq \text{graph}(T)$.
Recall that ${\cal M}$ denotes the set of maximally monotone operators on $\R^n$.
According to \cite[Fact 1]{RTBG18} and the monotone extension theorem \cite[Theorem 20.21]{BC17book}, the set
$\{(x_j,y_j)\}_{j\in K}$ is ${\cal M}$-interpolable if and only if
\[
\langle x_i - x_j,\ y_i - y_j\rangle \geq 0, \quad \forall i, j\in K.
\]

Let $S := \{(w^{k+1}, w^k - w^{k+1})\}_{k=0}^{N} \cup \{(0,0)\}$.
According to the PPA formula \eqref{PPA}, the procedure $\{w^1,\ldots,w^{N+1}\} = \text{PPA}(A,w^0,1,N+1)$ together with the requirement $0 \in A(0)$ in \eqref{PEP1} can be restated as $S \subseteq \text{graph}(A)$ for some $A\in {\cal P}$, i.e., the set $S$ is ${\cal P}$-interpolable. Since ${\cal P} \subseteq {\cal M}$, the set $S$ is also ${\cal M}$-interpolable. As a result, the following set of inequalities are satisfied:
\begin{subequations}\label{inter-conditions}
\begin{eqnarray}\label{inter-cond-1}
  \langle w^{i} - w^{j}, (w^{i-1} - w^{i}) - (w^{j-1} - w^{j})\rangle \geq 0, & \forall \, 1\leq i < j \leq N+1, \\
\label{inter-cond-2}
  \langle w^{i},  w^{i-1} - w^{i} \rangle \geq 0, & \forall \, 1\leq i\leq N+1.
\end{eqnarray}
\end{subequations}
On the other hand, if the inequalities in \eqref{inter-conditions} are satisfied, then the set $S$ is also ${\cal M}$-interpolable
\cite[Fact 1]{RTBG18}, \cite[Theorem 20.21]{BC17book}, i.e., there exists $A\in {\cal M}$ such that
$$S = \{(w^{k+1}, w^k - w^{k+1})\}_{k=0}^{N} \cup \{(0,0)\}\subseteq \text{graph}(A).$$
Since $(0,0) \in S$, it follows that $A\in {\cal P}$.
In summary, the infinite dimensional constraint $A \in {\cal P}$, as well as the conditions $0 \in A(0)$ and
$\{w^1,    \ldots, w^{N+1}\} = \text{PPA}(A,w^0,1,N+1)$ can be replaced by the set of inequalities given in \eqref{inter-conditions}.
As long as the dimension $n$ is larger than or equal to $N+2$ (the dimension of the Gram matrix in the next subsection), this replacement is no more than a reformulation while not a relaxation, which means that the exact worst case bound $\zeta(N)$ defined in \eqref{PEP1} can be computed, see, e.g., \cite[Theorem 5]{THG17mp}\cite[Proposition 2.6]{THG17siopt}\cite[Lemma 1]{RTBG18} for similar discussions.

\subsection{Grammian representation}\label{sc:gram}
In the following, we let $P = (w^0,w^1,\ldots,w^N,w^{N+1}) \in \R^{n\times (N+2)}$ and $X = P^TP \succeq 0$.
Then, $w^{i-1} = Pe_{i}$ for $i=1,2,\ldots,N+2$, where $e_{i}$
denotes the $i$-th unit vector of length $N+2$.

For $1\leq i < j \leq N+1$, we let $\xi_{ij} := (e_{i}-e_{i+1}) - (e_{j}-e_{j+1})$.
Then, the set of conditions in \eqref{inter-cond-1} can be equivalently stated as
\begin{eqnarray}\label{def:Aij}
\left\{
\begin{array}{l}
\langle w^{i} - w^{j},\ (w^{i-1} - w^{i}) - (w^{j-1} - w^{j})\rangle
=   \langle P(e_{i+1} - e_{j+1}),\ P\xi_{ij}\rangle  = \langle   A_{i,j},\ X\rangle/2 \geq 0,   \medskip\\
A_{i,j} :=  (e_{i+1}-e_{j+1})\xi_{ij}^T + \xi_{ij}(e_{i+1}-e_{j+1})^T, \quad 1\leq i < j \leq N+1.
\end{array}
\right.
\end{eqnarray}
Note that here $A_{i,j}$ is a symmetric matrix dependent on $i$ and $j$, instead of the $(i,j)$-th entry of matrix $A$. Though, it is a little bit abuse of notation, no confusion occurs.
%
On the other hand, the set of conditions in \eqref{inter-cond-2} can be restated as
\begin{eqnarray}\label{def:Bi}
\left\{
\begin{array}{l}
\langle w^{i},  w^{i-1} - w^{i} \rangle
=   \langle Pe_{i+1}, P(e_i - e_{i+1}) \rangle  = \langle   B_i, X\rangle/2 \geq 0,   \medskip\\
B_{i} :=  e_{i+1} (e_{i} - e_{i+1})^T + (e_{i} - e_{i+1})e_{i+1}^T, \quad 1\leq i \leq N+1.
\end{array}
\right.
\end{eqnarray}
Furthermore, it is apparent that the constraint $\|w^0\|=1$ is equivalent to
\begin{eqnarray}\label{def:E11}
\langle E_{11}, X\rangle := \langle e_1e_1^T,X \rangle = 1.
\end{eqnarray}
Finally, the objective function can be rewritten as
\begin{eqnarray}\label{def:C}
\|w^N -  w^{N+1}\|^2 = \langle (e_{N+1}-e_{N+2})(e_{N+1}-e_{N+2})^T, X       \rangle :=\langle C, X\rangle.
\end{eqnarray}

\subsection{The SDP reformulation}
With the discussions in sections \ref{sc:OperInter} and \ref{sc:gram}, when $n \geq N+2$, the PEP given in \eqref{PEP1} can be equivalently reformulated as the following SDP:
\begin{eqnarray}\label{SDP}
\max_{X\in {\mathbb S}^{N+2}} \left\{
\langle C, X\rangle
\left|
  \begin{array}{ll}
\langle   A_{i,j}, X\rangle \geq 0, & 1\leq i < j \leq N+1, \smallskip\\
\langle B_{i}, X \rangle \geq 0, & 1\leq  i \leq N+1, \smallskip \\
\langle E_{11}, X\rangle = 1, & X \succeq 0
  \end{array}
  \right.
  \right\},
\end{eqnarray}
where the matrices $\{A_{i,j}: 1\leq i < j \leq N+1\}$, $\{B_{i}: 1\leq i  \leq N+1\}$, $E_{11}$
and $C$ are defined, respectively, in \eqref{def:Aij}, \eqref{def:Bi}, \eqref{def:E11} and \eqref{def:C}.
When $n < N+2$, the SDP \eqref{SDP} is only a relaxation of PEP \eqref{PEP1}, in which case the optimum value of \eqref{SDP} is an upper bound of $\zeta(N)$. In summary, for any $n>0$, it suffices to consider SDP \eqref{SDP} to obtain an upper bound of $\zeta(N)$, which is our main purpose of considering PEP.

We use SeDuMi \cite{Stu99oms} to solve the above SDP for $N=1, 2, \ldots, 100$.
The worst case iteration bound just computed is compared in Figure~\ref{figpep} to the nonergodic convergence rate \cite[Proposition]{BL78}\cite[Theorem 3.1]{HY12c}, i.e., the one given in \eqref{bnd-HY12c}.

\begin{figure}
\centering\begin{tikzpicture}[scale=1.4]
\begin{axis}[xmin=0, xmax=100, xlabel={$N$}, ylabel={Performance measure factor}]
\addplot[mark=, blue] table[x=N,y=f1]{result.dat};
\addlegendentry{PEP}
\addplot[dashed, samples=100, domain=1:100] plot (\x, {1/(\x+1)});
\addlegendentry{$\frac{1}{N+1}$}
\end{axis}
\end{tikzpicture}
\caption{Comparison between the worst case bound computed by PEP (the lower curve) and the convergence rate of Douglas-Rachford operator splitting method (the dashed curve).
}
\label{figpep}
\end{figure}
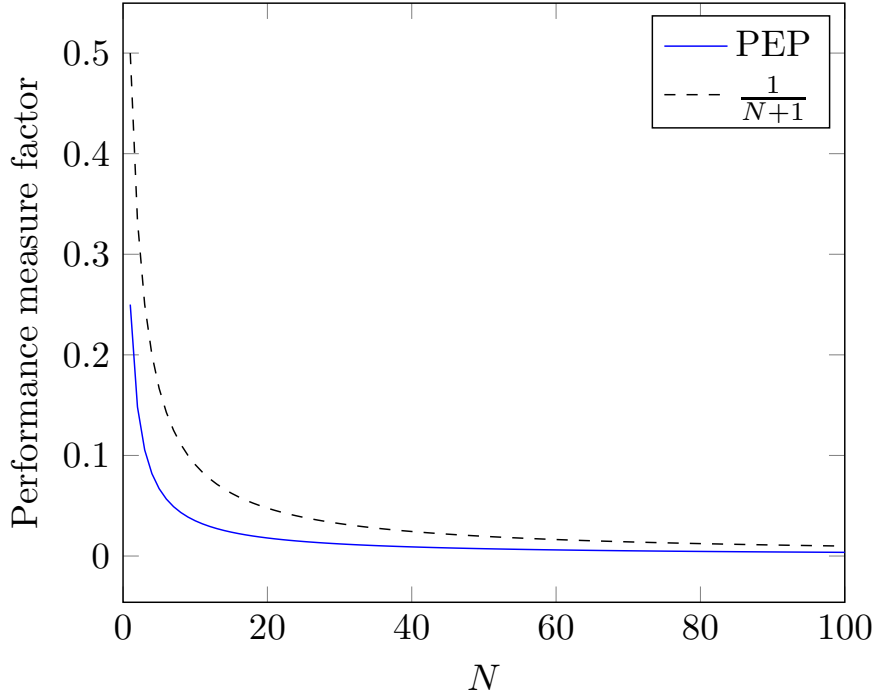

\section{An upper bound on $\zeta(N)$}\label{sc:upper_bound}
In this section, an upper bound on $\zeta(N)$ is given by considering the following Lagrangian dual of \eqref{SDP}:
%
%
\begin{eqnarray}\label{SDP-dual}
\min_{\{\lambda_{i,j}\}, \{\mu_i\}, \eta} \left\{
\eta \left|
  \begin{array}{l}
\eta E_{11} -C - \sum_{1\leq i<j\leq N+1}  \lambda_{i,j} A_{i,j} - \sum_{i=1}^{N+1} \mu_{i}  B_{i} \succeq 0, \smallskip \\
\lambda_{i,j}\geq 0, \quad 1\leq i<j\leq N+1, \smallskip \\
\mu_{i}\geq 0, \quad i=1,2,\ldots,N+1, \quad \eta\in\R
  \end{array}
\right.
\right\}.
\end{eqnarray}
Here, $\lambda_{i,j}$, $\mu_i$ and $\eta$ are Lagrange multipliers.
%
For convenience, we define
\begin{equation}\label{def:M}
M := \eta E_{11}-C - \sum_{1\leq i<j\leq N+1}  \lambda_{i,j} A_{i,j} - \sum\nolimits_{i=1}^{N+1} \mu_{i}  B_{i} \in \mathbb{S}^{N+2}.
\end{equation}
The idea is to construct a feasible solution to the dual SDP \eqref{SDP-dual} and thus obtain an upper bound on $\zeta(N)$. Indeed, we have successfully constructed one, which is given in the following proposition.

\begin{proposition}\label{prop:dual-feasible}
Let $\{A_{i,j}: 1\leq i<j \leq N+1\}$, $\{B_{i}: 1\leq i  \leq N+1\}$, $E_{11}$
and $C$ be defined, respectively, in \eqref{def:Aij}, \eqref{def:Bi}, \eqref{def:E11} and \eqref{def:C}.
For any integer $N > 0$, the following is a feasible solution to the dual SDP \eqref{SDP-dual}:
\begin{subequations}\label{def:dual-sol}
\begin{align}
\lambda_{i,j} &=
\begin{cases}
{N^{N - i} \over (N+1)^{N - i}} {i \over N+1},\quad & j-1=i =1, 2, \ldots, N,\\
0, &\text{otherwise},
\end{cases}\\
\mu_{i} &= {N^{N - i} \over (N+1)^{N - i}} {N - i \over (N+1)^2}, \quad i = 1, 2, \ldots, N, \\
\mu_{N+1} &= {1 \over N+1}, \\
\eta &= {N^{N} \over (N+1)^{N}} {1 \over N+1}.
\end{align}
\end{subequations}
\end{proposition}
\noindent Apparently, all the quantities in (\ref{def:dual-sol}a)-(\ref{def:dual-sol}c) are nonnegative.
It remains to show that the choices of the Lagrange multipliers specified in \eqref{def:dual-sol} are such that the matrix $M$ defined in \eqref{def:M} is positive semidefinite,
which is tedious and postponed to the Appendix.
Hence, it follows from weak duality that
\begin{equation}\label{zeta-N-ub}
\zeta(N) \leq \eta = {1 \over (1+{1\over N})^{N}} {1 \over N+1}, \quad \forall\, N >0.
\end{equation}
Therefore, we have obtained an upper bound on $\zeta(N)$.
We emphasize that this upper bound on $\zeta(N)$ is nonasymptotic in the sense that it holds for any $N>0$. Furthermore, \eqref{zeta-N-ub} holds for any $n>0$, where $n$ denotes the dimension of the underlying Euclidean space.

\section{A lower bound on $\zeta(N)$}\label{sc:example}
In order to obtain a lower bound on $\zeta(N)$, an example in $\R^2$ is constructed in this section.
Let $\theta\in (0, \pi/2)$, and define a rotation matrix in $\R^2$ as
\begin{equation}\label{def:theta}
\Theta :=
\begin{pmatrix}
  \cos\theta & -\sin\theta \\
  \sin\theta & \cos\theta
\end{pmatrix}.
\end{equation}
It is easy to check that $\Theta^T = \Theta^{-1}$ and for any integer $k$,
\begin{equation}\label{Theta_property}
   \Theta^k + \Theta^{-k} =
\begin{pmatrix}
  \cos(k\theta) & -\sin(k\theta) \\
  \sin(k\theta) & \cos(k\theta)
\end{pmatrix}
+
\begin{pmatrix}
  \cos(k\theta) & \sin(k\theta) \\
  -\sin(k\theta) & \cos(k\theta)
\end{pmatrix}
=
2\cos(k\theta) I.
\end{equation}
To simplify the notation, we denote $\beta = \cos\theta$.
The concrete example we construct is
\begin{alignat}{2}
\notag
w^0 &= (1,0)^T, \\
\notag
w^1 &= \beta \Theta w^0, &   u^1 &= w^0 - w^1,    \\
\notag
w^2 &= \beta \Theta w^1 = \beta^2 \Theta^2 w^0, & u^2 &= w^1 - w^2, \\
\label{example}
   &  \cdots  & &\cdots \\
\notag
w^{N+1} &= \beta \Theta w^N = \beta^{N+1} \Theta^{N+1} w^0,\qquad & u^{N+1} &= w^N - w^{N+1}, \\
\notag
w^* &= 0, & u^* &=  0.
\end{alignat}
The geometry of this example for $N=5$ and $N=100$ is shown in Figure~\ref{Fig:example}.
%
\begin{figure}[h!]
\begin{center}
\begin{tikzpicture}[scale=0.4]
\tikzset{help lines/.style={dotted}}
\draw[help lines, dotted, step=1] (-5, 0) grid (11, 9);
\coordinate (O) at (0, 0);
\coordinate (w0) at (10.00000000000000, 0);
\coordinate (w1) at (8.33333333333333, 3.72677996249965);
\coordinate (w2) at (5.55555555555556, 6.21129993749942);
\coordinate (w3) at (2.31481481481482, 7.24651659374932);
\coordinate (w4) at (-0.77160493827160, 6.90144437499935);
\coordinate (w5) at (-3.21502057613169, 5.46364346354116);
\coordinate (w6) at (-4.71536351165981, 3.35486879340247);
\draw [-] (w6) node [right] {$w^6$}-- (w5) node [above] {$w^5$};
\draw [-] (w5) -- (w4) node [above] {$w^4$};
\draw [-] (w4) -- (w3) node [above] {$w^3$};
\draw [-] (w3) -- (w2) node [above] {$w^2$};
\draw [-] (w2) -- (w1) node [right] {$w^1$};
\draw [-] (w1) -- (w0) node [right] {$w^0$};
\draw [dashed] (O) node [above] {$O$} -- (w0) -- cycle;
%
\draw [dashed, blue] (w1) -- (O) -- (w2);
\draw [dashed, blue, -] ($(w0)!2!(w1)$) -- (w1);
\draw [dashed, ultra thick, red, -latex] ($(w0)!2!(w1)$) -- (w2) node [right] {\small $u^1-u^2$};
\coordinate (v1) at ($(O)!0.95!(w1)$);
\coordinate (v2) at ($(v1)!1!270:(w1)$);
\coordinate (v3) at ($(w0)!(v2)!(w1)$);
\draw (v1) -- (v2) -- (v3);
\coordinate (vv1) at ($(O)!0.95!(w2)$);
\coordinate (vv2) at ($(vv1)!1!270:(w2)$);
\coordinate (vv3) at ($(w1)!(vv2)!(w2)$);
\draw (vv1) -- (vv2) -- (vv3);
\end{tikzpicture}
\begin{tikzpicture}[scale=0.63]
\begin{axis}[xmin=-1, xmax=1, xlabel={}, ylabel={}]
\addplot[mark=, thick, blue] table[x=x, y=y]{w100.dat};
\end{axis}
\end{tikzpicture}
\end{center}
\caption{Geometry of the constructed example. Left: $N=5$. Right: $N=100$.}\label{Fig:example}
\end{figure}
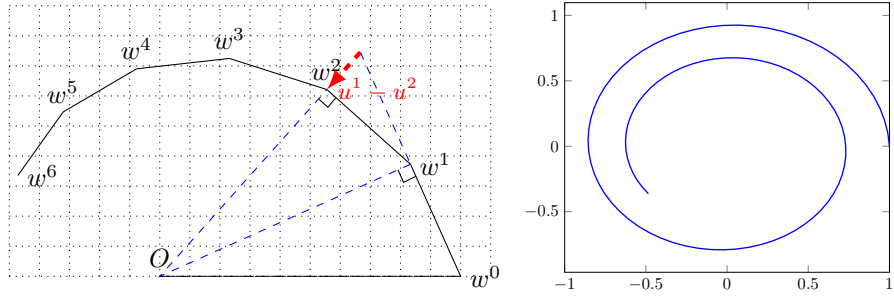
\begin{proposition}
Let $N>0$ be an integer and $w^0$, $\{ (w^k,u^k): k=1,\ldots,N+1\}$ and $(w^*,u^*)$  be defined in \eqref{example} with $\theta\in (0, \pi/2)$ and $\Theta$ given in \eqref{def:theta}.
Then, there hold
\begin{subequations}
   \label{example-inequalities}
\begin{eqnarray}
   \langle w^j - w^i, u^j - u^i\rangle = 0, &   \forall\, 1\leq i < j \leq N+1, \\
   \langle w^i - w^*, u^i - u^*\rangle = 0, &  \forall\, 1\leq i \leq N+1.
\end{eqnarray}
\end{subequations}
\end{proposition}
\begin{proof}
First, for any $1\leq i\leq N+1$, we have
\begin{align*}
\langle w^i - w^*, u^i - u^*\rangle
&=\langle w^i, u^i\rangle \\
&= \langle w^i, w^{i-1} - w^i \rangle \\
&= \langle \beta^i \Theta^i w^0, \beta^{i-1} \Theta^{i-1} (I - \beta \Theta) w^0 \rangle \\
&= \beta^{2i-1} \langle  \Theta w^0, (I - \beta \Theta) w^0 \rangle \\
&= \beta^{2i-1} \left(\langle  \Theta w^0, w^0 \rangle - \beta  \langle  \Theta w^0, \Theta  w^0 \rangle\right)\\
&= \beta^{2i-1} \left(\langle  \Theta w^0, w^0 \rangle - \beta   \right)\\
&= 0,
\end{align*}
where in the third ``$=$'' we used the fact that $\Theta^T = \Theta^{-1}$, while the last ``$=$'' follows from $\langle  \Theta w^0, w^0 \rangle = \|\Theta w^0\|\|\theta\| \cos \theta= \beta$.
Then, for $1\leq i < j \leq N+1$,  we have
\begin{align*}
& \langle w^j - w^i, u^j - u^i \rangle \\
=&  \langle \beta^{j-i} \Theta^{j-i} w^i - w^i, \beta^{j-i} \Theta^{j-i} u^i - u^i \rangle \\
=&  \langle \beta^{j-i} \Theta^{j-i} w^i, \beta^{j-i} \Theta^{j-i} u^i \rangle - \langle \beta^{j-i} \Theta^{j-i} w^i,   u^i \rangle - \langle w^i, \beta^{j-i} \Theta^{j-i} u^i\rangle + \langle   w^i,  u^i \rangle \\
=&  \beta^{2(j-i)} \langle w^i, u^i \rangle - \beta^{j-i} \langle (\Theta^{j-i} + \Theta^{i-j}) w^i, u^i \rangle   + \langle    w^i,  u^i \rangle \\
=&   (\beta^{2(j-i)} + 1) \langle w^i, u^i \rangle - \beta^{j-i} \langle (\Theta^{j-i} + \Theta^{i-j}) w^i, u^i \rangle     \\
=& (\beta^{2(j-i)} + 1) \langle w^i, u^i \rangle - 2\beta^{j-i} \cos( (j-i)\theta)   \langle w^i, u^i \rangle      \\
=&   0,
\end{align*}
where the last but one ``$=$'' follows from \eqref{Theta_property} and the last  ``$=$'' is because
$\langle w^i, u^i\rangle =0$ for $1\leq i \leq N+1$.
\end{proof}

It follows from \eqref{example-inequalities}, \cite[Fact 1]{RTBG18} and \cite[Theorem 20.21]{BC17book} that there exists a maximally monotone operator $A\in {\cal M}$ such that the set $\{(w^k,u^k): k=1,2,\ldots,N+1\} \cup \{(w^*,u^*)\}$ defined in \eqref{example} satisfies
\[
\{(w^k,u^k): k=1,2,\ldots,N+1\} \cup \{(w^*,u^*)\} \subseteq \text{graph}(A),
\]
i.e., $\{(w^k,u^k): k=1,2,\ldots,N+1\} \cup \{(w^*,u^*)\}$ is ${\cal M}$-interpolable.  Since $u^k = w^{k-1} - w^k\in A(w^k)$ for $k=1,2,\ldots,N+1$ and $A$ is maximally monotone, this implies that $\{w^k: k=1,2,\ldots,N+1\}$ obeys the PPA scheme
\[
w^{k+1} = J_A(w^{k}), \quad \forall\, k=0,\ldots, N.
\]
In hindsight, we found that $A$ admits an explicit formula given by
$$A(w) = \tan(\theta)
\begin{pmatrix}
0 & 1 \\
-1 & 0 \\
\end{pmatrix}
w, \quad w\in\R^2.
$$
Apparently, $A$ is continuous and satisfies $\langle Aw, w\rangle = 0$ for any $w\in\R^2$, which is sufficient to guarantee that $A$ is maximally monotone. Furthermore, it is easy to show that $J_A  =  \beta \Theta$.
Finally, we recall that the fixed-point residual $e(w,\lambda)$ is defined in \eqref{def:e}.
Then, it is elementary to verify that
\begin{align*}
\varepsilon(A,w^0,w^*,1,N) &= {\|e(w^N,1)\|^2 \over \|w^0 - w^*\|^2} =  \|w^N - w^{N+1}\|^2  \\
&= \|\beta^N  \Theta^N w^0 - \beta^{N+1} \Theta^{N+1} w^0\|^2 \\
&= \beta^{2N} \langle w^0 - \beta \Theta w^0, w^0 - \beta \Theta w^0 \rangle \\
&= \beta^{2N} (1 -2 \beta \langle w^0, \Theta w^0\rangle + \beta^2 \langle \Theta w^0, \Theta w^0 \rangle)  \\
&= \beta^{2N} (1 - \beta^2).
\end{align*}
By maximizing the above function over $\beta$, we derive that $\beta=\cos\theta=\sqrt{N\over N+1}$ and $$\varepsilon(A,w^0,w^*,1,N)={1 \over \big(1+{1\over N}\big)^N} {1 \over N+1}.$$
In summary, we have constructed a two-dimensional example, which gives the following lower bound
\begin{equation}
  \label{zeta-N-lb}
\zeta(N) \geq {1 \over (1+{1\over N})^N} {1 \over N+1},\quad \forall\, N >0.
\end{equation}
Since the constructed example lives in $\R^2$, the lower bound inequality given in \eqref{zeta-N-lb} is valid for any $n\geq 2$.
%
%
%
%

\section{Concluding remarks}\label{sc:conclusions}
PPA is a fundamental algorithm framework in convex optimization as well as monotone operator theory.
By using the performance estimation framework recently developed in  \cite{DT14mp,THG17mp,THG17siopt,dKGT17ol,THG18jota}, we established an upper bound \eqref{zeta-N-ub} on the exact worst case convergence rate $\zeta(N)$. We also constructed a two-dimensional example to provide a lower bound \eqref{zeta-N-lb} on $\zeta(N)$. Since the upper bound in \eqref{zeta-N-ub} coincides with the lower bound in  \eqref{zeta-N-lb}, we thus have shown that the exact worst case convergence rate $\zeta(N)$ is exactly ${1 \over (1+{1\over N})^N} {1 \over N+1}$  for any $n\geq 2$.
This also implies that the dual feasible solution constructed in \eqref{def:dual-sol} is in fact optimal for the dual SDP \eqref{SDP-dual}. Since $\lambda_{i,j}=0$ for $j-i\geq 2$, this hints that the constraints  $\{\langle   A_{i,j}, X\rangle \geq 0: 1\leq i < j \leq   N, \, j-i\geq 2\}$ are likely not active.
Apparently, ${1 \over (1+{1\over N})^N} {1 \over N+1}$ is only an upper bound on $\zeta(N)$ for $n=1$.
The established bound is nonasymptotic since it holds for all integers $N>0$.
Although the improvement of the obtained bound over the known bound \eqref{bnd-HY12c} is only approximately a constant factor of $\exp(-1)$, it is optimal in terms both the order as well as the constants involved.
Undoubtedly, this sharpens our understanding to the fundamental PPA.


\def\cprime{$'$}

\begin{appendix}
\section{Proof of Proposition \ref{prop:dual-feasible}}
This appendix is devoted to prove that the matrix $M$ defined by \eqref{def:M}-\eqref{def:dual-sol} is positive semidefinite.
From the definitions of $\{A_{i,j}: 1\leq i<j\leq N+1\}$, $\{B_i: 1\leq i\leq N+1\}$ and $C$ in \eqref{def:Aij}, \eqref{def:Bi} and \eqref{def:C}, respectively, we have
\begin{subequations}\label{ABC-nz}
  \begin{align}
A_{i,i+1}(i:i+2,i:i+2) &=
\begin{pmatrix}
0 & 1 & -1\\
1 & -4 & 3\\
-1 & 3 & -2
\end{pmatrix}, \quad i = 1,2,\ldots,N, \\
B_i(i:i+1,i:i+1) &=
\begin{pmatrix}
    0 & 1 \\
    1 & -2 \\
\end{pmatrix},
i = 1,2,\ldots, N+1, \\
C(N+1:N+2,N+1:N+2) &=
\begin{pmatrix}
    1 & -1 \\
    -1 & 1 \\
\end{pmatrix}.
\end{align}
\end{subequations}
It is then not difficult to verify from \eqref{def:M}, \eqref{def:dual-sol} and \eqref{ABC-nz} that $M$ is a pentadiagonal matrix whose nonzero entries are given by
\begin{eqnarray}\label{def:M-elements}
\left\{
\begin{array}{l}
M_{11} = \eta = {N^N \over (N+1)^N} {1 \over N+1}, \smallskip \\
M_{ii} =
\left\{
  \begin{array}{ll}
    4\lambda_1 + 2\mu_1, & \hbox{$i=2$}, \smallskip \\
    2\lambda_{i-2} + 4\lambda_{i-1} + 2\mu_{i-1}, & \hbox{$i=3,\ldots,N$},
  \end{array}
\right.  \smallskip\\
\;\;\;\quad = {N^{N+1-i} \over (N+1)^{N+1-i}} \frac{(i-1)(6N+2)}{(N+1)^2}, \quad i = 2,3,\ldots,N, \smallskip \\
M_{N+1,N+1} = 2\lambda_{N-1} + 4\lambda_{N} + 2\mu_{N} -1 = {5N^2 - 1 \over (N+1)^2}, \smallskip \\
M_{N+2,N+2} = 2\lambda_{N} + 2\mu_{N+1} -1 = 1, \smallskip \\
M_{1,2} = M_{2,1}  = -\lambda_1-\mu_1 = - {N^N \over (N+1)^N} {2 \over N+1}, \smallskip \\
M_{i,i+1} = M_{i+1,i} = -3\lambda_{i-1}-\lambda_{i} -\mu_{i} =  - {N^{N+1-i} \over (N+1)^{N+1-i}} {4i-2 \over  N+1 },   \quad i=2,3,\ldots,N, \smallskip\\
M_{N+1,N+2} = M_{N+2,N+1}   =   -3 \lambda_N -\mu_{N+1} + 1 = -{2N \over N+1}, \smallskip \\
M_{i,i+2} = M_{i+2,i} =  \lambda_{i} = {N^{N-i} \over (N+1)^{N-i}} {i \over N+1},  \quad i=1,2,\ldots,N.
\end{array}
\right.
\end{eqnarray}
If $N=1$, then
\[
M =
\begin{pmatrix}
    {1/4} & -{1/2} & {1/2} \\
    -{1/2} & 2 & -3 \\
    {1/2} & -3 & 8 \\
\end{pmatrix},
\]
which is positive definite. In the following, we assume that $N>1$.
Our proof of $M\succeq 0$ heavily relies on the following Schur complement lemma.
\begin{lemma}[Schur complement, {\cite{Hayn68}}]\label{lem:schur}
Let $k>0$ be an integer, $H \in \mathbb{S}^k$ be a real symmetric matrix, $b \in \mathbb{R}^k$ be a column vector and $c>0$ be a scalar.
Then
  \[
   \begin{pmatrix}
      H & b \\
      b^T & c \\
  \end{pmatrix}
   \succeq 0
  \text{~~if and only if~~}
   H - {bb^T\over c} \succeq 0.
  \]
\end{lemma}
\noindent
Schur complement lemma has several variants and generalizations, and Lemma \ref{lem:schur} is only one special case that
we will use repeatedly.
For this purpose, we fix our notation first.
Applying the Schur complement reduction as specified in Lemma \ref{lem:schur} to $M \in \mathbb{S}^{N+2}$ recursively for $j$ times, we obtain a matrix of order $N+2 - j$, which we denote by $M^{[N+2-j]} \in \mathbb{S}^{N+2 - j}$, $j = 0,1,\ldots,N-1$. Note that here the superscript $k$ in  $M^{[k]}$ corresponds to the order of $M^{[k]}$.
Since $M$ is a pentadiagonal matrix, applying Schur complement reduction as specified in Lemma \ref{lem:schur} to $M^{[k]}$ once is equivalent to deleting the $k$-th row and the $k$-th column of it and meanwhile updating $M^{[k]}(k-2:k-1,k-2:k-1)$ appropriately.
Specifically,  $M^{[k]}$ and $M^{[k-1]}$ appear as
\begin{equation}
\label{schur5}
M^{[k]} =
\begin{pmatrix}
       * & * & * &   &   &   \\
       * & \ddots & \ddots & \ddots &   &   \\
       * & \ddots & * & * & * &   \\
         & \ddots & * & \times & \times & \alpha \\
         &   & * & \times & \times & \beta \\
         &   &   & \alpha & \beta & \gamma \\
\end{pmatrix}_{k\times k} \text{\ and\ }
M^{[k-1]} =
\begin{pmatrix}
       * & * & * &   &       \\
       * & \ddots & \ddots & \ddots &       \\
       * & \ddots & * & * & *     \\
         & \ddots & * & \triangle & \triangle \\
         &   & * &  \triangle & \triangle  \\
\end{pmatrix}_{(k-1)\times (k-1)}
\end{equation}
where $\alpha = M^{[k]}_{k-2,k}$, $\beta = M^{[k]}_{k-1,k}$, $\gamma = M^{[k]}_{k,k}$, and
\[
M^{[k-1]}(k-2:k-1,k-2:k-1) =
\begin{pmatrix}
    M^{[k]}_{k-2,k-2} - {\alpha^2 \over \gamma} & M^{[k]}_{k-2,k-1} - {\alpha\beta\over \gamma} \smallskip \\
    M^{[k]}_{k-1,k-2} - {\alpha\beta\over \gamma} & M^{[k]}_{k-1,k-1} - {\beta^2 \over \gamma} \\
\end{pmatrix}.
\]
Note that all the $*$'s in both $M^{[k]}$ and $M^{[k-1]}$ are never touched and remain the same as those corresponding elements in the original matrix $M$.

For $j=0$, we have
\[
M^{[N+2-j]} = M^{[N+2]} = M =
\begin{pmatrix}
    M(1:N+1,1:N+1)  & M(1:N+1,N+2) \smallskip \\
    M(1:N+1,N+2)^T & M_{N+2,N+2} \\
\end{pmatrix}.
\]
Since $M_{N+2,N+2} = 1 > 0$, according to   Lemma \ref{lem:schur}, $M^{[N+2]}\succeq 0$ if and only if
\[
M^{[N+1]} := M(1:N+1,1:N+1) -  { M(1:N+1,N+2) M(1:N+1,N+2)^T \over M_{N+2,N+2}} \succeq 0.
\]
In view of (\ref{schur5}), and direct computation indicate that
\begin{align}\label{formula_(N+1)}
\nonumber
&
{M^{[N+1]}(:,\ N-1:N+1) }=M^{[N+1]}(N-1:N+1,\ :)^T \\
\nonumber
=&
{
\left(
  \begin{array}{lll}
  \multicolumn{3}{c}{O_{N-3,3}}\medskip\\
   M_{N-3,N-1} & 0 & 0\\
    M_{N-2,N-1} & M_{N-2,N} & 0 \medskip \\
    M_{N-1,N-1} & M_{N-1,N} & M_{N-1,N+1} \\
    M_{N,N-1}   & M_{N,N} - {M_{N,N+2}^2 \over M_{N+2,N+2}}  & M_{N,N+1} - {M_{N,N+2}M_{N+1,N+2} \over M_{N+2,N+2}}  \\
    M_{N+1,N-1} & M_{N,N+1} - {M_{N,N+2}M_{N+1,N+2} \over M_{N+2,N+2}}  & M_{N+1,N+1} -  {M_{N+1,N+2}^2 \over M_{N+2,N+2}}  \\
  \end{array}
\right)} \\
=&
{
\left(
  \begin{array}{lll}
  \multicolumn{3}{c}{O_{N-4,3}}\medskip\\
   {N^{3} \over (N+1)^{3}} {N-3 \over N+1} & 0 & 0\\
   - {N^3 \over (N+1)^3}{(4N-10) \over N+1}
   & {N^{2} \over (N+1)^{2}} {N-2 \over N+1}
   & 0 \smallskip\\
 {N^2 \over (N+1)^2} \frac{(N-2)(6N+2)}{(N+1)^2}
   & - {N^2 \over (N+1)^2}   {(4N-6) \over N+1}
   & {N \over N+1}  {N-1 \over N+1} \smallskip \\
   - {N^2 \over (N+1)^2}   {(4N-6) \over N+1} & {N\over N+1}{5N^2 - 5N - 2 \over (N+1)^2}
   & -{N(2N-2) \over (N+1)^2}  \smallskip  \\
    {N \over N+1}  {N-1 \over N+1}& -{N(2N-2) \over (N+1)^2}   & {N-1 \over N+1} \\
  \end{array}
\right).
}
\end{align}
It follows from $N>1$ that $M^{[N+1]}_{N+1,N+1} = {N-1 \over N+1} > 0$. Thus, $M^{[N+1]}\succeq 0$ if and only if  $M^{[N]}\succeq 0$.
In the following, we apply Schur complement reduction as specified in Lemma \ref{lem:schur} to $M^{[k]}$  and show inductively for $k=N+1, N, \ldots, 5$ that
$M^{[k]}_{k,k} > 0$ and
 \begin{align}\label{formula_(k)}
 \nonumber
& M^{[k]}(:,\ k-2:k)=M^{[k]}(k-2:k,\ :)^T\\
=&
\left(
  \begin{array}{lll}
  \multicolumn{3}{c}{O_{k-5,3}}\medskip\\
  {N^{N+4-k} \over (N+1)^{N+4-k}} {k-4 \over N+1} & 0 & 0\\
    - {N^{N+4-k} \over (N+1)^{N+4-k}} {(4k-14) \over N+1}
    & {N^{N+3-k} \over (N+1)^{N+4-k}} {k-3 \over N+1}
    & 0 \smallskip\\
    %
    {N^{N+3-k} \over (N+1)^{N+3-k}} \frac{(k-3)(6N+2)}{(N+1)^2}
      & - {N^{N+3-k} \over (N+1)^{N+3-k}}   {(4k-10) \over N+1}
      & {N^{N+2-k} \over (N+1)^{N+2-k}} {k-2 \over N+1} \smallskip \\
    %
   - {N^{N+3-k} \over (N+1)^{N+3-k}}   {(4k-10) \over N+1}
   & {N^{N+2-k} \over (N+1)^{N+2-k}} { (5k-11)N + (k-3) \over (N+1)^2}
   & -   {N^{N+2-k} \over (N+1)^{N+2-k}}  {2k-4 \over N+1}  \smallskip  \\
   %
   {N^{N+2-k} \over (N+1)^{N+2-k}}  {k-2 \over N+1} & -  {N^{N+2-k} \over (N+1)^{N+2-k}}  {2k-4 \over N+1}  & {N^{N+1-k} \over (N+1)^{N+1-k}}{k-2 \over N+1} \\
  \end{array}
\right).
\end{align}
It is already noted that $M^{[N+1]}_{N+1,N+1} > 0$.
 By comparing with \eqref{formula_(N+1)}, we see that \eqref{formula_(k)} holds for $k=N+1$.
Assume that \eqref{formula_(k)}  holds for $5 < k \leq N+1$. We will show that \eqref{formula_(k)} also holds for $k-1$.
Since $k>5$, we have  $M^{[k]}_{k,k} = {N^{N+1-k} \over (N+1)^{N+1-k}}{k-2 \over N+1} > 0$.
Furthermore, direct calculations show that
 \begin{align*}\label{formula_(k-1)}
 \nonumber
& M^{[k-1]}(:,\ k-3:k-1)= M^{[k-1]}(k-3:k-1, :)^T\\
\nonumber
=&
\left(
  \begin{array}{lll}
  \multicolumn{3}{c}{O_{k-6,3}}\medskip\\
   M_{k-5,k-3} & 0 & 0\\
    M_{k-4,k-3} & M_{k-4,k-2} & 0 \\
    M_{k-3,k-3} & M_{k-3,k-2} & M_{k-3,k-1} \\
     M_{k-2,k-3} & M_{k-2,k-2} - {M_{k-2,k}^2\over M^{[k]}_{k,k}}   &  M_{k-2,k-1} - { M_{k-2,k}M^{[k]}_{k-1,k} \over M^{[k]}_{k,k}}  \\
   M_{k-1,k-3} & M_{k-1,k-2} - {M_{k-2,k}M^{[k]}_{k-1,k}\over M^{[k]}_{k,k}} &   M^{[k]}_{k-1,k-1} - {(M^{[k]}_{k-1,k})^2 \over M^{[k]}_{k,k}}
  \end{array}
\right)\\
\nonumber
 =&
\left(
  \begin{array}{lll}
  \multicolumn{3}{c}{O_{k-6,3}}\medskip\\
  {N^{N+5-k} \over (N+1)^{N+5-k}} {k-5 \over N+1} & 0 & 0\\
 - {N^{N+5-k} \over (N+1)^{N+5-k}} {(4k-18) \over N+1}
    & {N^{N+4-k} \over (N+1)^{N+4-k}} {k-4 \over N+1}
    & 0 \smallskip\\
 {N^{N+4-k} \over (N+1)^{N+4-k}} \frac{(k-4)(6N+2)}{(N+1)^2}
      & - {N^{N+4-k} \over (N+1)^{N+4-k}}  {(4k-14) \over N+1}
      & {N^{N+3-k} \over (N+1)^{N+3-k}}  {k-3 \over N+1} \smallskip \\
 - {N^{N+4-k} \over (N+1)^{N+4-k}}   {(4k-14) \over N+1}
   & {N^{N+3-k} \over (N+1)^{N+3-k}}{ (5k-16)N + (k-4) \over (N+1)^2}
   & -  {N^{N+3-k} \over (N+1)^{N+3-k}}  {2k-6 \over N+1}  \smallskip  \\
  {N^{N+3-k} \over (N+1)^{N+3-k}} {k-3 \over N+1} & - {N^{N+3-k} \over (N+1)^{N+3-k}}  {2k-6 \over N+1}  & {N^{N+2-k} \over (N+1)^{N+2-k}}{k-3 \over N+1} \\
  \end{array}
\right).
\end{align*}
That is, \eqref{formula_(k)} also holds for $k-1$. In summary, \eqref{formula_(k)} holds for all $k = N+1,N,\ldots,5$, and $M\succeq 0$ if and only if $M^{[5]}\succeq 0$. Since $M^{[5]}_{55} = {N^{N-4} \over (N+1)^{N-4}} {3 \over N+1} >0$, $M^{[5]}\succeq 0$ if and only if
\[
M^{[4]} =
{
\left(
  \begin{array}{llll}
    {N^{N} \over (N+1)^{N}} {1 \over N+1}
   & - {N^{N} \over (N+1)^{N}} {2\over N+1}
   & {N^{N-1} \over (N+1)^{N-1}} {1\over N+1}
   & 0 \smallskip \\
    - {N^{N} \over (N+1)^{N}} {2\over N+1}
    & {N^{N-1} \over (N+1)^{N-1}} \frac{(6N+2)}{(N+1)^2}
    & - {N^{N-1} \over (N+1)^{N-1}} {6 \over N+1}
    & {N^{N-2} \over (N+1)^{N-2}} {2 \over N+1} \smallskip\\
    {N^{N-1} \over (N+1)^{N-1}} {1\over N+1}
    & - {N^{N-1} \over (N+1)^{N-1}} {6 \over N+1}
    & {N^{N-2} \over (N+1)^{N-2}} \frac{9N+1}{(N+1)^2}
      & - {N^{N-2} \over (N+1)^{N-2}}   {4 \over N+1}  \smallskip \\
   0
   & {N^{N-2} \over (N+1)^{N-2}} {2 \over N+1}
   & - {N^{N-2} \over (N+1)^{N-2}}   {4 \over N+1}
   &  {N^{N-3} \over (N+1)^{N-3}} { 2 \over N+1}   \\
  \end{array}
\right)\succeq 0.
}
\]
Since $M^{[4]}_{4,4}>0$, $M^{[4]}\succeq 0$ if and only if $M^{[3]}\succeq 0$.
One more step computation shows that
\[
M^{[3]} =
{
\left(
  \begin{array}{lllll}
    {N^{N} \over (N+1)^{N}} {1 \over N+1}
   & - {N^{N} \over (N+1)^{N}} {2\over N+1}
   & {N^{N-1} \over (N+1)^{N-1}} {1\over N+1}  \smallskip \\
    - {N^{N} \over (N+1)^{N}} {2\over N+1}
    & {N^{N-1} \over (N+1)^{N-1}} \frac{4N}{(N+1)^2}
    & - {N^{N-1} \over (N+1)^{N-1}} {2 \over N+1}  \smallskip\\
    {N^{N-1} \over (N+1)^{N-1}} {1\over N+1}
    & - {N^{N-1} \over (N+1)^{N-1}} {2 \over N+1}
    & {N^{N-2} \over (N+1)^{N-2}} \frac{1}{N+1}   \\
  \end{array}
\right).
}
\]
Since $M^{[3]}_{3,3}>0$, $M^{[3]}\succeq 0$ if and only if $M^{[2]}\succeq 0$.
One further step computation shows that $M^{[2]} = O_{2,2}$, the two by two zero matrix,
which is positive semidefinite. In summary, we have shown that, for any integer $N>0$, $M$ is  positive semidefinite  by recursively using the Schur complement reduction procedure in Lemma \ref{lem:schur}.
\end{appendix}

\end{document}